\documentclass[11pt]{article}
\usepackage{amsfonts, amsthm, amssymb, amsmath, stmaryrd, mathtools}
\usepackage{mathrsfs,array}
\usepackage{xy}
\usepackage{hyperref}
\input xy
\xyoption{all}

\def\from{\colon}

\DeclareMathOperator{\Spa}{Spa}

\DeclareMathOperator{\Spf}{Spf}

\DeclareMathOperator{\Proj}{Proj}
\DeclareMathOperator{\ad}{ad}

\DeclareMathOperator{\nr}{nr}
\DeclareMathOperator{\Gal}{Gal}
\DeclareMathOperator{\Hom}{Hom}
\DeclareMathOperator{\tr}{tr}
\DeclareMathOperator{\Spec}{Spec}
\def\isom{\cong}
\def\et{\text{\'et}}
\def\C{\mathbb{C}}

\def\Q{\mathbb{Q}}

\def\Z{\mathbb{Z}}
\def\FF{\mathbb{F}}
\def\OO{\mathcal{O}}
\def\injects{\hookrightarrow}

\DeclareMathOperator{\Frob}{Frob}

\newcommand{\abs}[1]{\left\lvert #1 \right\rvert}
\newcommand{\class}[1]{\left< #1 \right>}
\newcommand{\set}[1]{\left\{ #1 \right\}}

\newcommand{\tatealgebra}{\class}
\newcommand{\powerseries}[1]{\llbracket #1 \rrbracket}
\newcommand{\laurentseries}[1]{(\!(#1)\!)}

\numberwithin{equation}{subsection}
\newtheorem{Theorem}{Theorem}
\numberwithin{Theorem}{subsection}
\newtheorem{lemma}[Theorem]{Lemma}

\newtheorem{prop}[Theorem]{Proposition}
\newtheorem{conj}[Theorem]{Conjecture}

\theoremstyle{definition}
\newtheorem{defn}[Theorem]{Definition}

\newtheorem{exmp}[Theorem]{Example}

\title{$\Gal(\overline{\Q}_p/\Q_p)$ as a geometric fundamental group}
\author{Jared Weinstein}

\begin{document}
\maketitle

\section{Introduction}  Let $p$ be a prime number.  In this article we present a theorem, suggested by Peter Scholze, which states that $\Gal(\overline{\Q}_p/\Q_p)$ is the \'etale fundamental group of certain object $Z$ which is {\em defined over an algebraically closed field}.  As a consequence, $p$-adic representations of $\Gal(\overline{\Q}_p/\Q_p)$ correspond to $\Q_p$-local systems on $Z$.

The precise theorem involves perfectoid spaces, \cite{ScholzePerfectoidSpaces}.  Let $C/\Q_p$ be complete and algebraically closed.  Let $D$ be the open unit disk centered at 1, considered as a rigid space over $C$, and given the structure of a $\Z_p$-module where the composition law is multiplication, and $a\in \Z_p$ acts by $x\mapsto x^a$.  Let
\[ \tilde{D}=\varprojlim_{x\mapsto x^p} D. \]
Then $\tilde{D}$ is no longer a classical rigid space, but it does exist in Huber's category of adic spaces, and is in fact a perfectoid space.    Note that $\tilde{D}$ has the structure of a $\Q_p$-vector space.  Let $\tilde{D}^*=\tilde{D}\backslash\set{0}$;  this admits an action of $\Q_p^\times$.

\begin{Theorem} \label{MainTheoremForQp} The category of $\Q_p^\times$-equivariant finite \'etale covers of $\tilde{D}^*$ is equivalent to the category of finite \'etale $\Q_p$-algebras.
\end{Theorem}

The object $Z$ of the first paragraph is then the quotient $Z=\tilde{D}^*/\Q_p^\times$.  This quotient doesn't belong to the category of adic spaces.  Rather, one has a Yoneda-style construction.  The category $\text{Perf}_C$ of perfectoid spaces over $C$ has a pro-\'etale topology, \cite{ScholzePAdicHodgeTheory}, and one has a sheaf of sets $Z$ on $\text{Perf}_C$, namely the sheafification of $X\mapsto  \Hom(X,\tilde{D}^*)/\Q_p^\times$.  Thus $Z$ belongs to the category of sheaves of sets on $\text{Perf}_C$ which admit a surjective map from a representable sheaf.  In this category, a morphism $\mathcal{F}'\to \mathcal{F}$ is called finite \'etale if the pull-back to any representable sheaf is finite \'etale in the usual sense.   Then one may define the fundamental group for such objects, and Thm. \ref{MainTheoremForQp} asserts that  $\pi_1^{\et}(Z)=\Gal(\overline{\Q}_p/\Q_p)$.

Theorem \ref{MainTheoremForQp} can be generalized to a finite extension $E/\Q_p$.  Let $\pi\in E$ be a uniformizer, and let $H$ be the corresponding Lubin-Tate formal $\OO_E$-module.  Then $D$ gets replaced by the generic fiber $H^{\ad}_C$;  this is the unit disc centered at 0, considered as an adic space over $C$.  $H^{\ad}_C$ is endowed with the $\OO_E$-module structure coming from $H$.  Form the {\em universal cover} $\tilde{H}^{\ad}_{C}=\varprojlim H^{\ad}_{C}$, where the inverse limit is taken with respect to multiplication by $\pi$ in $H$.  Then $\tilde{H}^{\ad}_C$ is an $E$-vector space object in the category of perfectoid spaces over $C$.  Form the quotient
\[ Z_E = (\tilde{H}^{\ad}_C\backslash\set{0})/E^\times. \]
Then the main theorem (Thm. \ref{MainTheorem}) states that the categories of finite \'etale covers of $Z_E$ and $\Spec E$ are equivalent, so that $\pi^{\et}_1(Z_E)=\Gal(\overline{E}/E)$.  

The proof hinges on a combination of two themes:  the {\em fundamental curve $X$ of $p$-adic Hodge theory}, due to Fargues-Fontaine, and the {\em tilting equivalence}, due to Scholze.  Let us sketch the proof in the case $E=\Q_p$.  Let $C^\flat$ be the tilt of $C$, a perfectoid field in characteristic $p$.  Consider the punctured open disc $D^*_{C^{\flat}}$ (with parameter $t$) and its universal cover $\tilde{D}^*_{C^\flat}$.  Then $\tilde{D}^*_{C^\flat}$ is simultaneously a perfectoid space over two fields:
\[
\xymatrix{
& \Spa \FF_p\laurentseries{t^{1/p^\infty}} \\
\tilde{D}_{C^{\flat}} \ar[ur] \ar[dr] & \\
& \Spa C^{\flat}
}
\]
Considered as a perfectoid space over $C^{\flat}$, $\tilde{D}_{C^{\flat}}^*$ has an obvious ``un-tilt'', namely $\tilde{D}_C^*$.  The other field $\FF_p\laurentseries{t^{1/q^\infty}}$ is the tilt of $\hat{\Q}_p(\mu_{p^\infty})$.  It turns out that there is a perfectoid space over $\hat{\Q}_p(\mu_{p^\infty})$ whose tilt is also $\tilde{D}_C^*$, and here is where the Fargues-Fontaine curve comes in.

The construction of the Fargues-Fontaine curve $X$ is reviewed in \S\ref{TheCurve}.  $X$ is an integral noetherian scheme of dimension 1 over $\Q_p$, whose closed points parametrize un-tilts of $C^{\flat}$ modulo Frobenius.  For our purposes we need the adic version $X^{\ad}$, which is the quotient of another adic space $Y^{\ad}$ by a Frobenius automorphism $\phi$.  The extension of scalars $Y^{\ad}\hat{\otimes}\hat{\Q}_p(\mu_{p^\infty})$ is a perfectoid space;  by a direct calculation (Prop. \ref{HTildeVersusY}) we show that its tilt is isomorphic to $\tilde{D}^*_{C^\flat}$.  In this isomorphism, the action of $\Gal(\Q_p(\mu_{p^\infty})/\Q_p)\isom \Z_p^\times$ on the field of scalars $\hat{\Q}_p(\mu_{p^\infty})$ corresponds to the geometric action of $\Z_p^\times$ on $\tilde{D}^*_{C^\flat}$, and the automorpism $\phi$ corresponds (up to absolute Frobenius) to the action of $p$ on $\tilde{D}^*_{C^\flat}$.  

Therefore under the tilting equivalence, finite \'etale covers of $\tilde{D}_C^*$ and $Y^{\ad}\hat{\otimes}\hat{\Q}_p(\mu_{p^\infty})$ are identified.  The same goes for $\tilde{D}_C^*/p^\Z$ and $X^{\ad}\hat{\otimes}\hat{\Q}_p(\mu_{p^\infty})$.  Now we apply the key fact that $X^{\ad}$ is geometrically simply connected.  The same statement is proved in \cite{FarguesFontaineCurve} for the algebraic curve $X$;  we have adapted the proof for $X^{\ad}$ in Prop. \ref{GeometricallySimplyConnected}.  Thus finite \'etale covers of $\tilde{D}_C^*/p^\Z$ are equivalent to finite \'etale $\Q_p(\mu_{p^\infty})$-algebras.  Now we can descend to $\Q_p$:  $\Z_p^\times$-equivariant finite \'etale covers of $\tilde{D}_C^*/p^\Z$ are equivalent to finite \'etale $\Q_p$-algebras, which is Thm. \ref{MainTheoremForQp}.

This entire article is an elaboration of comments made to me by Peter Scholze.  I also thank Laurent Fargues for many helpful remarks.  The author is supported by NSF Award DMS-1303312.

\section{The Fargues-Fontaine curve}
\label{TheCurve}

\subsection{The rings $B$ and $B^+$}  Here we review the construction of the Fargues-Fontaine curve.  The construction requires the following two inputs:
\begin{itemize}
\item A finite extension $E/\Q_p$, with uniformizer $\pi$ and residue field $\FF_q/\FF_p$,
\item A perfectoid field $F$ of characteristic $p$ containing $\FF_q$.
\end{itemize}

Note that a perfectoid field $F$ of characteristic $p$ is the same as a perfect field of characteristic $p$ which is complete with respect to a nontrivial (rank 1) valuation.  Let $x\mapsto \abs{x}$ denote one such valuation.

Let $W(F)$ be the ring of Witt vectors of the perfect field $F$, and let $W_{\OO_E}(F)=W(F)\otimes_{W(\FF_q)} \OO_E$.  A typical element of $W_{\OO_E}(F)$ is a series
\[ x = \sum_{n\gg -\infty} [x_n]\pi^n, \]
where $x_n\in F$.  Let $B^b\subset W_{\OO_E}(F)[1/\pi]$ denote the subalgebra defined by the condition that $\abs{x_n}$ is bounded as $n\to\infty$.  Equivalently, $B^b=W_{\OO_E}(\OO_F)[1/\pi,1/[\varpi]]$, where $\varpi\in F$ is any element with $0<\abs{\varpi}<1$.   Then $B^b$ clearly contains $B^{b,+}=W(\OO_F)\otimes_{W(k)} \OO_E$ as a subring.

For every $0<\rho<1$ we define a norm $\abs{\;}_\rho$ on $B^b$ by
\[ \abs{x}_{\rho}=\sup_{n\in\Z} \abs{x_n} \rho^n. \]

\begin{defn} Let $B$ denote the Fr\'echet completion of $B^{b}$ with respect to the family of norms $\abs{\;}_\rho$, $\rho\in (0,1)$.
\end{defn}

$B$ can be expressed as the inverse limit $\varprojlim_I B_I$ of Banach $E$-algebras, where $I$ ranges over closed subintervals of $(0,1)$, and $B_I$ is the completion of $B^b$ with respect to the norm
\[ \abs{x}_I=\sup_{\rho\in I}\abs{x}_\rho. \]
(Note that $\abs{x}_\rho$ is continuous in $\rho$ and therefore bounded on $I$.)

It will be useful to give an explicit description of $B_I$.

\begin{lemma}  Let $I\subset (0,1)$ be a closed subinterval whose endpoints lie in the value group of $F$:  $I=[\abs{\varpi_1},\abs{\varpi_2}]$.  Let $B^{b,\circ}_I$ be subring of elements $x\in B^b$ with $\abs{x}_I \leq 1$.  Then   $B^{b,\circ}_I:=B^{b,+}\left[\frac{[\varpi_1]}{\pi},\frac{\pi}{[\varpi_2]}\right]$.
\end{lemma}

\begin{proof} It is easy to see that
\[ \abs{\frac{[\varpi_1]}{\pi}}_I = \abs{\frac{\pi}{[\varpi_2]}} = 1, \]
which proves that $B^{b,+}\left[\frac{[\varpi_1]}{\pi},\frac{\pi}{[\varpi_2]}\right]\subset B^{b,\circ}_I$.

Conversely, suppose $x=\sum_{n\gg -\infty} [x_n]\pi^n\in B^{b,\circ}_I$.
Then $\abs{x_n}\leq \abs{\varpi_1}^{-n}$ for $n<0$ and $\abs{x_n}\leq \abs{\varpi_2}^{-n}$ for $n\geq 0$.  This shows that the ``tail term" $\sum_{n<0} [x_n]\pi^n$ lies in $B^{b,+}[[\varpi_1]/\pi]$.  It also shows that for $n\geq 0$, each term $[x_n]\pi^n$ lies in $B^{b,+}[\pi/[\varpi_2]]$.   Since $x\in B^b$, there exists $C>0$ with $\abs{x_n}\leq C$ for all $n$.  Let $N$ be large enough so that  $\abs{\varpi_2}^NC\leq 1$;  then for $n\geq N$ we have $\abs{\varpi_2^Nx_n}\leq 1$.  Then the sum of the terms of $x$ of index $\geq N$ are
\[ \pi^N([x_N]+[x_{N+1}]\pi+\dots)=\frac{\pi^N}{\varpi_2^N}([\varpi^Nx_N]+[\varpi^Nx_{N+1}]\pi+\dots), \]
which lies in $B^{b,+}\left[\frac{[\varpi_1]}{\pi},\frac{\pi}{[\varpi_2]}\right]$.
\end{proof}

\begin{lemma} If $B_I^{\circ}$ is the closed unit ball in $B_I$, then  $B_I^{\circ}$ is the $\pi$-adic completion of $B^{b,\circ}_I$.
\end{lemma}

\begin{proof}  $B_I$ is the completion of $B^b$ with respect to the norm $\abs{\;}_I$.  This norm induces the $\pi$-adic topology on $B^{b,\circ}_I$, so that the $\pi$-adic completion of $B^{b,\circ}_I$ is the same as the closed unit ball in $B_I$.
\end{proof}

%%Too easy to be a lemma?

%%Other possible results to include:  construction of B^+, and B^{phi=pi^d} = (B^+)^{\phi = pi^d}.  This last is Prop. 1.13 and the remark that follows in in \cite{FarguesFontaineDurham}

The $q$th power Frobenius automorphism of $F$ induces a map $\phi\from B^{b,+}\to B^{b,+}$.  For $\rho\in (0,1)$ we have $\abs{x^\phi}_\rho = \abs{x}_{\rho^{1/q}}^q$.  As a result, $\phi$ induces an endomorphism of $B$.  We have $B^{\phi=1}=E$.  (See Prop. 7.1 of \cite{FarguesFontaineCurve}.) We put
\[ P=\bigoplus_{d\geq 0} B^{\phi=\pi^d}, \]
a graded $E$-algebra.

\begin{defn}  Let $X_{E,F}=\Proj P$.  
\end{defn}

\begin{Theorem}{\cite{FarguesFontaineCurve}, Th\'eor\`eme 10.2}
\label{XIsACurve}
\begin{enumerate}
\item $X_{F,E}$ is an integral noetherian scheme which is regular of dimension 1.
\item $H^0(X_{F,E},\OO_{X_{F,E}})=E$. ($E$ is the field of definition of $X_{F,E}$)
\item For a finite extension $E'/E$ there is a canonical isomorphism $X_{F,E'}\isom X_{F,E}\otimes_E E'$.
\end{enumerate}
\end{Theorem}

%%%maybe this part is optional.
%\begin{Theorem}
%The map $x\mapsto (k(x),F\injects k(x)^{\flat})$ is a bijection from $\abs{X}$ onto the set of equivalence classes of pairs $(F^\sharp,\iota)$, where $F^\sharp$ is a perfectoid field containing $E$ and $\iota\from F\to (F^\sharp)^{\flat}$ is a continuous embedding of fields which presents $(F^\sharp)^{\flat}$ as a finite extension of $F$.
%\end{Theorem}

\section{The adic curve $X_{F,E}^{\ad}$}

\subsection{Generalities on formal schemes and adic spaces}
The category of adic spaces is introduced in $\cite{HuberAdicSpaces}$.  In brief, an adic space is a topological space $X$ equipped with a sheaf of topological rings $\OO_X$, which is locally isomorphic to $\Spa(R,R^+)$.  Here $(R,R^+)$ is an ``affinoid ring'', $R^+\subset R$ is an open and integrally closed subring, and $\Spa(R,R^+)$ is the space of continuous valuations $\abs{\;}$ on $R$ with $\abs{R^+}\leq 1$.

\begin{defn} A topological ring $R$ is {\em{f-adic}} if it contains an open subring $R_0$ whose topology is generated by a finitely generated ideal $I\subset R_0$.
%Such an $R_0$ (resp., an $I$) is called a {\em ring of definition} (resp., {\em ideal of definition}) of $R$.
An {\em affinoid ring} is a pair $(R,R^+)$, where $R$ is $f$-adic and $R^+\subset R$ is an open and integrally closed subring consisting of power-bounded elements.  A morphism of affinoid algebras $(R,R^+)\to (S,S^+)$ is a continuous homomorphism $R\to S$ which sends $R^+\to S^+$.

$R$ is a {\em Tate ring} if it contains a topologically nilpotent unit.
\end{defn}

Note that if $(R,R^+)$ and $R$ is Tate, say with topologically nilpotent $\pi \in R$, then $R=R^+[1/\pi]$ and $\pi R\subset R$ is open.  Note also that if $(R,R^+)$ is an affinoid algebra over $(E,\OO_E)$, then $R$ is Tate.

It is important to note that a given affinoid ring $(R,R^+)$ does not necessarily give rise to an adic space $\Spa(R,R^+)$, because the structure sheaf on $\Spa(R,R^+)$ is not necessarily a sheaf.  Let us say that $(R,R^+)$ is {\em sheafy} if the structure presheaf on $\Spa(R,R^+)$ is a sheaf.  Huber shows $(R,R^+)$ is a sheaf when $R$ is ``strongly noetherian'', meaning that $R\tatealgebra{X_1,\dots,X_n}$ is noetherian for all $n$.  In \S2 of \cite{ScholzeWeinsteinModuli} we constructed a larger category of ``general'' adic spaces, whose objects are sheaves on the category of complete affinoid rings (this category can be given the structure of a site in an obvious way).   If $(R,R^+)$ is a (not necessarily sheafy) affinoid ring, then $\Spa(R,R^+)$ belongs to this larger category.  If $X$ is an adic space in the general sense, let us call $X$ an {\em honest adic space} if it belongs to the category of adic spaces in the sense of Huber;  {\em i.e.} if it is locally $\Spa(R,R^+)$ for a sheafy $(R,R^+)$.

%\begin{lemma}  Suppose that $(R,R^+)$ is an affinoid algebra such that $R$ is a Tate ring.  Let $\pi\in R^+$ be a topologically nilpotent element which is invertible in $R$.  Then $R=R^+[1/\pi]$, and the topology on $R^+$ is induced by the ideal $\pi R^+$, where $\pi\in R^+$ is a topologically nilpotent element which is invertible in $R$.
%\end{lemma}
%
%\begin{lemma} \label{RIsPiAdic}  Let $f\in R$.  Since $\pi$ is topologically nilpotent and $R^+\subset R$ is open, we have $\pi^n f\in R^+$ for some $n\geq 1$.  Thus $R=R^+[1/\pi]$.
%
%Suppose that $I$ is a finitely-generated ideal of $R^+$ which generates its topology.  After replacing $\pi$ with some power we may assume that $\pi\in I$.   Suppose $f\in I$ is any element.   Since $f^n\to 0$ in $R$, and $\pi\in R$ is a unit, $\pi^{-1} f^n\to 0$ as well.  Since $R^+\subset R$ is open, $\pi^{-1} f^n\in R^+$ for some $n\geq 1$.  Thus $f^n\in \pi R^+$.  Since $I$ is finitely generated, this argument shows that $I^N\in \pi R^+$ for some $N\geq 1$.  Since $\pi R^+\subset  I$, we have that $\pi R^+$ and $I$ generate the same topology.
%\end{lemma}

Now suppose $R$ is an $\OO_E$-algebra which is complete with respect to the topology induced by a finitely generated ideal $I\subset R$ which contains $\pi$.  Then $\Spf R$ is a formal scheme over $\Spf \OO_E$, and $(R,R)$ is an affinoid ring.  One can form the (general) adic space $\Spa(R,R)$, which is fibered over the two-point space $\Spa(\OO_E,\OO_E)$.

By \cite{ScholzeWeinsteinModuli}, Prop. 2.2.1, $\Spf R\mapsto \Spa(R,R)$ extends to a functor $M\mapsto M^{\ad}$ from the category of formal schemes over $\Spf \OO_E$ locally admitting a finitely generated ideal of definition, to the category of (general) adic spaces over $\Spa(\OO_E,\OO_E)$.

\begin{defn} Let $M$ be a formal scheme over $\Spf \OO_E$ which locally admits a finitely generated ideal of definition.  The {\em adic generic fiber} $M^{\ad}_{\eta}$ is the fiber of $M^{\ad}$ over the generic point $\eta=\Spa(E,\OO_E)$ of $\Spa(\OO_E,\OO_E)$.  We will also notate this as $M^{\ad}_E$.  
\end{defn}

\begin{lemma} \label{ExhaustionByRn} Let $R$ be a flat $\OO_E$-algebra which is complete with respect to the topology induced by a finitely generated ideal $I$ containing $\pi$.  Let $f_1,\dots,f_r,\pi$ be generators for $I$.  For $n\geq 1$, let $S_n=R[f_1^n/\pi,\dots,f_r^n/\pi]^{\wedge}$ ($\pi$-adic completion).  Let $R_n=S_n[1/\pi]$, and let $R_n^+$ be the integral closure of $S_n$ in $R_n$.  There are obvious inclusions $R_{n+1}\injects R_n$ and $R_{n+1}^+\injects R_n^+$.  Then
\[ M^{\ad}_E = \varinjlim_n \Spa(R_n,R_n^+).\]
\end{lemma}

\begin{proof}
This amounts to showing that whenever $(T,T^+)$ is a complete affinoid algebra over $(E,\OO_E)$, then
\[ M^{\ad}_E(T,T^+) = \varinjlim_n \Hom((R_n,R_n^+),(T,T^+)). \]
An element of the left hand side is a continuous $\OO_E$-linear homomorphism $g\from R\to T^+$;  we need to produce a corresponding homomorphism $R_n^+\to T^+$.  We have that $g(f_i)\in T^+$ is topologically nilpotent for $i=1,\dots,r$, since $f_i$ is.  Since $\pi T^+$ is open in $T$, there exists $n\geq $ so that $g(f_i)^n\in \pi T^+$, $i=1,\dots,r$.  Thus $g$ extends to a map $R^+[f_1^n/\pi,\dots,f_r^n/\pi]\to T^+$.  Passing to the integral closure of the completion gives a map $R_n\to T^+$ as required.
\end{proof}

\begin{exmp} \label{RigidOpenDisc} The adic generic fiber of $R=\Spf \OO_E\powerseries{t}$ is the rigid open disc over $E$.
\end{exmp}

%%Maybe a word here on perfectoid spaces, preperfectoid spaces, and strong completion -- these all come up later

\subsection{Formal schemes with perfectoid generic fiber}
\label{PerfectoidFormalSchemes}
In this section $K$ is a perfectoid field of characteristic 0, and $\varpi\in K$ satisfies $0<\abs{\varpi}\leq \abs{p}$.   Assume that $\varpi=\varpi^{\flat}$ for some $\varpi^\flat\in \OO_{K^{\flat}}$, so that $\varpi^{1/p^n}\in \OO_K$ for all $n\geq 1$.

\begin{defn} Let $S$ be a ring in characteristic $p$.  $S$ is {\em semiperfect} if the Frobenius map $S\to S$ is surjective.  If $S$ is semiperfect, let $S^{\flat}=\varprojlim_{x\mapsto x^p} S$, a perfect topological ring.

If $R$ is a topological $\OO_K$-algebra with $R/\varpi$ semiperfect, then write $R^\flat=(R/\varpi)^\flat$, a perfect topological $\OO_{K^\flat}$-algebra.
\end{defn}

If $R$ is complete with respect to a finitely generated ideal of definition, then so is $R^\flat$.

\begin{prop} \label{TiltOfGenericFiber} Let $R$ be an $\OO_K$-algebra which is complete with respect to a finitely generated ideal of definition.  Assume that $R/\pi$ is a semiperfect ring.  Then $(\Spf R)^{\ad}_{K}$ and $(\Spf R^{\flat})^{\ad}_{K^\flat}$ are a perfectoid spaces over $K$ and $K^\flat$ respectively, so that in particular they are honest adic spaces.  There is a natural isomorphism of perfectoid spaces over $K^\flat$:
\[ (\Spf R)_{K}^{\ad,\flat}\isom (\Spf R^{\flat})_{K^{\flat}}^{\ad} \]
\end{prop}

\begin{proof}  First we show that $(\Spf R^{\flat})^{\ad}_{\eta^\flat}$ is a perfectoid space over $K^\flat$.  Let $f_1,\dots,f_r$ be generators for an ideal of definition of $R^\flat$.  Then $(\Spf R^{\flat})_{\eta^\flat}^{\ad}$ is the direct limit of subspaces $\Spa(R_n^{\flat},R_n^{\flat,+})$ as in Lemma \ref{ExhaustionByRn}.  Here $R_n^{\flat}=R^{\flat}\tatealgebra{f_i^n/\varpi^\flat}[1/\varpi^\flat]$.  Let
\[ S_n^{\flat,\circ} = R^{\flat}\tatealgebra{\left(\frac{f_1^n}{\varpi^\flat}\right)^{1/p^\infty},\cdots,
\left(\frac{f_r^n}{\varpi^\flat}\right)^{1/p^\infty}} \subset R_n^{\flat}\]
Then $S_n^{\flat,\circ}$ is a $\varpi^\flat$-adically complete flat $\OO_{K^\flat}$-algebra.  We claim that $S_n^{\flat,\circ a}$ is a perfectoid $\OO_{K^\flat}^{a}$-algebra, which is to say that the Frobenius map $S_n^{\flat,\circ}/\varpi^{1/p}\to S_n^{\flat,\circ}/\varpi$ is an almost isomorphism.  For this we refer to \cite{ScholzePerfectoidSpaces}, proof of Lemma 6.4(i) (case of characteristic $p$).  This shows that $R_n^{\flat}=S_n^{\flat,\circ}[1/\varpi^\flat]$ is a perfectoid $K^\flat$-algebra, and thus $\Spf(R_n^{\flat},R_n^{\flat,+})$ is a perfectoid affinoid, and in particular it is an honest adic space.  

To conclude that $(\Spf R^\flat)^{\ad}_{K^\flat}$ is a perfectoid space, one needs to show that its structure presheaf is a sheaf.  (I thank Kevin Buzzard for pointing out this subtlety.)  But since $(\Spf R^\flat)^{\ad}_{K^\flat}$ is the direct limit of the $\Spa(R_n^{\flat},R_n^{\flat,+})$, its structure presheaf is the inverse limit of the pushforward of the structure presheaves of the $\Spa(R_n^{\flat},R_n^{\flat,+})$, which are all sheaves.  Now we can use the fact that an arbitrary inverse limit of sheaves is again a sheaf.

We now turn to characteristic 0.  Recall the sharp map $f\mapsto f^\sharp$, which is a map of multiplicative monoids $R^\flat\to R$.  The elements $f_1^\sharp,\dots,f_r^\sharp$ generate an ideal of definition of $R$.   Then $(\Spf R)_{\eta}^{\ad}$ is the direct limit of subspaces $\Spa(R_n,R_n^+)$, where $R_n=R\tatealgebra{f_i^n/\varpi}[1/\varpi]$.  The proof of \cite{ScholzePerfectoidSpaces}, Lemma 6.4(i) (case of characteristic $0$) shows that $\Spa(R_n,R_n^+)$ is a perfectoid affinoid, and Lemma 6.4(iii) shows that the tilt of $\Spa(R_n,R_n^+)$ is $\Spa(R_n^{\flat},R_n^{\flat,+})$.
\end{proof}

\subsection{The adic spaces $Y_{F,E}^{\ad}$ and $X_{F,E}^{\ad}$}
Once again, $E/\Q_p$ is a finite extension.

A scheme $X/E$ of finite type admits a canonical analytification $X^{\ad}$:  this is a rigid space together with a morphism $X\to X^{\ad}$ of locally ringed spaces satisfying the appropriate universal property.   GAGA holds as well:  If $X$ is a projective variety, then the categories of coherent sheaves on $X$ and $X^{\ad}$ are equivalent.

The Fargues-Fontaine curve $X=X_{F,E}$, however, is not of finite type, and it is not obvious that such an analytification exists.  However, in \cite{FarguesQuelquesResultats} there appears an adic space $X^{\ad}$ defined over $E$, such that $X$ and $X^{\ad}$ satisfy a suitable formulation of GAGA.

\begin{defn} \label{DefnOfYad} Give $B^{b,+}=W_{\OO_E}(\OO_F)$ the $I$-adic topology, where $I=([\varpi],\pi)$ and $\varpi\in \OO_F$ is an element with $0<\abs{\varpi}<1$.  Note that $B^{b,+}$ is $I$-adically complete.  Let
\[Y^{\ad}=(\Spf B^{b,+})^{\ad}_{E}\backslash\set{0},\]
where ``$0$'' refers to the valuation on $B^{b,+}=W_{\OO_E}(\OO_F)$ pulled back from the $\pi$-adic valuation on $W_{\OO_E}(k_F)$.
\end{defn}

\begin{prop} As (general) adic spaces we have
\[ Y^{\ad} = \varinjlim_I \Spa(B_I,B_I^\circ). \]
\end{prop}

This shows that our definition of $Y^{\ad}$ agrees with the definition in \cite{FarguesQuelquesResultats},
D\'efinition 2.5.

\begin{proof}  By Prop. \ref{ExhaustionByRn} we have
\[ (\Spf B^{b,+})^{\ad}_{E} = \varprojlim_{\varpi_1} \Spa(R_{\varpi_1},R_{\varpi_1}^+),\]
where $\varpi_1$ runs over elements of $\OO_F$ with $0<\abs{\varpi}<1$, $R_{\varpi_1}^+$ is the $\pi$-adic completion of $B^{b,+}\left[\frac{[\varpi_1]}{\pi}\right]$, and $R_{\varpi_1}=R_{\varpi_1}^+[1/\pi]$.

By definition, $Y^{\ad}$ is the complement in $(\Spf B^{b,+})^{\ad}_{E}$ of the single valuation pulled back from the $\pi$-adic valuation on $W_{\OO_E}(k)$.  Thus if $x\in Y^{\ad}$ we have $\abs{[\varpi](x)}\neq 0$ for all $\varpi\in \OO_F\backslash\set{0}$.   Thus there exists $\varpi_2\in\OO_F$ with $\abs{\varpi_2(x)}\geq \abs{\pi(x)}$.  On the other hand we have $x\in \Spa(R_{\varpi_1},R_{\varpi_1}^+)$ for some $\varpi_1$.  Thus $x$ belongs to the rational subset $\Spa(B_I,B_I^\circ)\subset \Spa(R_{\varpi_1},R_{\varpi_1}^+)$, where $I=[\abs{\varpi_1},\abs{\varpi_2}]$.  This shows that the $\Spa(B_I,B_I^{\circ})$ cover $Y^{\ad}$.  
%could possibly add more details here.
\end{proof}

The following conjecture is due to Fargues.

\begin{conj} \label{BINoetherian} Let $I\subset (0,1)$ be a closed interval.  The Banach algebra $B_I$ is {\em strongly noetherian}.  That is, for every $n\geq 1$, $B_I\tatealgebra{T_1,\dots,T_n}$ is noetherian.
\end{conj}

Despite not being able to prove Conj. \ref{BINoetherian}, we have the following proposition (Th\'eor\`eme 2.1 of \cite{FarguesQuelquesResultats}), which is proved by extending scalars to a perfectoid field and appealing to results of \cite{ScholzePerfectoidSpaces}.

\begin{prop} $Y^{\ad}$ is an honest adic space.
\end{prop}

\begin{defn} The adic space $X^{\ad}$ is the quotient of $Y^{\ad}$ by the automorphism $\phi$.
\end{defn}

\subsection{The adic Fargues-Fontaine curve in characteristic $p$}
The entire story of the Fargues-Fontaine curve can be retold when $E$ is replaced by a local field of residue field $\FF_q$.  The construction is very similar, except that $B^{b,+}=W(\OO_F)\otimes_{W(\FF_q)}
\OO_E$ must be replaced by $\OO_F\hat{\otimes}_{\FF_q}\OO_E$.  One arrives at a curve $X_{F,E}$ defined over $E$ satisfying the same properties as Thm. \ref{XIsACurve}.

One also gets spaces $X_{F,E}^{\ad}$ and $Y_{F,E}^{\ad}$.  Namely, 
\[ Y_{F,E}^{\ad}=\Spf \left(\OO_F\hat{\otimes}_{\FF_q}\OO_E\right)^{\ad}_{E}\backslash\set{0}, \]
where ``$0$'' refers to the pullback of the valuation on $\OO_F$ through the quotient map $\OO_F\hat{\otimes}_{\FF_q}\OO_E\to \OO_F$.  As before, $X_{F,E}^{\ad}$ is defined as the quotient of $Y_{F,E}^{\ad}$ by the automorphism $\phi$ coming from the Frobenius on $F$.  

Since $E\isom \FF_q\laurentseries{t}$, we have 
\[ Y_{F,E}^{\ad}=\Spf \OO_F\powerseries{t}^{\ad}_{F}\backslash\set{0}. \]
This is nothing but the punctured rigid open disc $D^{*}_F$.  We have the quotient $X_{F,E}^{\ad}=Y_{F,E}^{\ad}/\phi^\Z$.  Note that since $\phi$ does not act $F$-linearly, $X_{F,E}^{\ad}$ does not make sense as a rigid space over $F$.

\subsection{Tilts}  
Suppose once again that $E$ has characteristic 0.  Let $K$ be a perfectoid field containing $E$.  Let $Y^{\ad}_{F,E}\hat{\otimes} K$ be the strong completion of the base change to $K$ of $Y^{\ad}$.  Then
\[ Y^{\ad}_{F,E}\hat{\otimes} K = \Spf \left(W_{\OO_E}(\OO_F) \hat{\otimes}_{W(\FF_q)} \OO_K\right)^{\ad}_K
\backslash\set{0}. \]
Note that $(W_{\OO_E}(\OO_F)\hat{\otimes}_{W(\FF_q)} \OO_K)/\pi =\OO_F\otimes_{\FF_q} \OO_K/\pi$ is semiperfect, and $(W_{\OO_E}(\OO_F)\hat{\otimes}_{W(\FF_q)} \OO_K)^{\flat}=\OO_F\hat{\otimes}_{\FF_q}\OO_{K^\flat}$.  By Prop. \ref{TiltOfGenericFiber}, $Y^{\ad}_{F,E}\hat{\otimes} K$ is a perfectoid space, and
\begin{equation}
\label{TiltOfYad}
 (Y^{\ad}_{F,E}\hat{\otimes} K)^\flat \isom \Spf \left(\OO_F\hat{\otimes}_{\FF_q}\OO_{K^{\flat}}\right)^{\ad}_{K^{\flat}}\backslash\set{0}.
\end{equation}

As a special case, let $E_n$ be the field obtained by adjoining the $\pi^n$-torsion in a Lubin-Tate formal group over $E$, and let $E_\infty=\bigcup_{n\geq 1}E_n$.  Then $\hat{E}_\infty$ is a perfectoid field.  Let $L(E)$ be the imperfect field of norms for the extension $E_\infty/E$.  As a multiplicative monoid we have 
\[ L(E) = \varprojlim E_n, \]
where the inverse limit is taken with respect to the norm maps $E_{n+1}\to E_n$.  $L(E)\isom \FF_q\laurentseries{t}$ is a local field, and $\hat{E}_\infty\isom \FF_q\laurentseries{t^{1/q^\infty}}$ is the completed perfection of $L(E)$.
%%cite stuff?

The following proposition follows immediately from Eq. \eqref{TiltOfYad}.
\begin{prop} \label{TiltOfYadNormField}
\[ (Y^{\ad}_{F,E}\hat{\otimes}\hat{E}_\infty)^{\flat} \isom Y^{\ad}_{F,L(E)}\hat{\otimes}\hat{E}_{\infty}^{\flat}. \]
\end{prop}
This is Thm. 2.7(2) of \cite{FarguesQuelquesResultats}.  

\subsection{Classification of vector bundles on $X_{F,E}$}
In this section we review the results of \cite{FarguesFontaineCurve} concerning the classification of vector bundles on $X_{F,E}$.  

Recall that $X_{F,E}=\Proj P$, where $P$ is the graded ring $\bigoplus_{d\geq 0} (B^+)^{\phi=\pi^d}$.  For $d\in\Z$, let $P[d]$ be the graded $P$-module obtained from $P$ by shifting degrees by $d$, and let $\OO_{X_E}(d)$ be the corresponding line bundle on $X_E$.  For $h\geq 1$, let $E_h/E$ be the unramified extension of degree $h$, let $\pi_h\from X_{F,E_h}=X_{F,E}\otimes E_h\to X_{F,E}$ be the projection.  If $(d,h)=1$, define $\OO_{X_{F,E}}(d/h)=\pi_{h*}\OO_{X_{F,E_h}}(d)$, a vector bundle on $\OO_{X_{F,E}}$ of rank $h$.  One thus obtains a vector bundle $\OO_{X_{F,E}}(\lambda)$ for any $\lambda\in\Q$, which satisfy $\OO_{X_{F,E}}(\lambda)\otimes\OO_{X_{F,E}}(\lambda')\isom \OO_{X_{F,E}}(\lambda+\lambda')$.

\begin{prop}\label{GlobalSectionsOfOlambdaOnX} Let $\lambda\in\Q$.  Then $H^0(\lambda)\neq 0$ if and only if $\lambda\geq 0$.
\end{prop}

\begin{Theorem}\label{ClassificationOfVectorBundles} Every vector bundle $\mathcal{F}$ on $X_{F,E}$ is isomorphic to one of the form $\bigoplus_{i=1}^{s} \OO_{X_{F,E}}(\lambda_i)$, with $\lambda_i\in\Q$.  Furthermore, $\mathcal{F}$ determines the $\lambda_i$ up to permutation.
\end{Theorem}

\subsection{Classification of vector bundles on $X_{F,E}^{\ad}$}
In the absence of the noetherian condition of Conj. \ref{BINoetherian}, one does not have a good theory of coherent sheaves on $X_{F,E}^{\ad}$ when $E$ has characteristic 0.  This frustrates attempts to prove an analogue of Thm. \ref{ClassificationOfVectorBundles} for $X_{F,E}^{\ad}$.  Nonetheless, in \cite{FarguesQuelquesResultats}, Fargues gives an ad hoc notion of vector bundle on $X_{F,E}^{\ad}$, and proves a GAGA theorem relating vector bundles on $X_{F,E}$ to those on $X_{F,E}^{\ad}$.

When $E$ has characteristic $p$, however, $Y_{F,E}^{\ad}$ is isomorphic to the punctured disc $D^*=(\Spf \OO_F\powerseries{t})^{\ad}_{F}\backslash\set{0}$ over $F$.  In particular it is a rigid space, which does have a well-behaved theory of coherent sheaves.  In this case we are in the setting of \cite{HartlPinkVectorBundles}, which classifies $\phi$-equivariant vector bundles $\mathcal{F}$ on the punctured disc $D^*=Y_{F,E}^{\ad}$.  In \cite{HartlPinkVectorBundles}, these are called $\sigma$-bundles, and they admit a Dieudonn\'e-Manin classification along the lines of Thm. \ref{ClassificationOfVectorBundles}.  A $\sigma$-bundle is one and the same thing as a vector bundle on the quotient $Y_{F,E}^{\ad}/\phi^\Z=X_{F,E}^{\ad}$.  One gets a vector bundle $\OO(\lambda)$ on $X_{F,E}^{\ad}$ for every $\lambda\in\Q$.  

For the convenience of the reader we state the main definitions and constructions of \cite{HartlPinkVectorBundles}.

%In this language, the main theorems of \cite{HartlPinkVectorBundles} (Thm. 11.1 and Cor. 11.8) may be stated as follows.

%\begin{Theorem} \label{ClassificationOfSigmaBundles} Every vector bundle $\mathcal{F}$ on $X_{F,E}^{\ad}$ is isomorphic to one of the form $\bigoplus_{i=1}^{s} \OO(\lambda_i)$, with $\lambda_i\in\Q$.  Furthermore, $\mathcal{F}$ determines the $\lambda_i$ up to permutation.
%\end{Theorem}

A {\em vector bundle} on $D^*$ is by definition a locally free coherent sheaf of $\OO_{D^*}$-modules.  The global sections functor is an equivalence between the category of vector bundles on $D^*$ and the category of finitely generated projective modules over $\OO_{D^*}(D^*)$.  
The Frobenius automorphism $x\mapsto x^q$ on $F$ induces an automorphism $\sigma\from D^*\to D^*$ (the arithmetic Frobenius).  
Note that the $F$-points of $D^*$ are $\set{x\in F\vert 0<\abs{x}<1}$, and on this set $\sigma$ acts as $x\mapsto x^{q^{-1}}$.

%% In Hartl-Pink, \sigma is denoted \sigma_{D^*} %%

\begin{defn} A {\em $\sigma$-bundle} on $D^*$ is a pair $(\mathcal{F},\tau_{\mathcal{F}})$, where $\mathcal{F}$ is a vector bundle on $D^*$ and $\tau_{\mathcal{F}}$ is an isomorphism $\sigma^*\mathcal{F}\tilde{\to}\mathcal{F}$.  If $(\mathcal{F},\tau_{\mathcal{F}})$ is a $\sigma$-bundle, then a {\em global section} of $\mathcal{F}$ is a global section of $\mathcal{F}$ which is invariant under $\tau_{\mathcal{F}}$.
\end{defn}

We will refer to the pair $(\mathcal{F},\tau_{\mathcal{F}})$ simply as $\mathcal{F}$.  Then the space of global sections of $\mathcal{F}$ will be denoted $H^0(\mathcal{F})$.

The trivial vector bundle $\OO=\OO_{D^*}$ is a $\sigma$-bundle whose global sections $H^0(\OO)$ are the ring of power series $\sum_{n\in \Z} a_nt^n$ which converge on $D^*$ and which satisfy $a_n^q=a_n$ (thus $a_n\in\FF_q$).  This ring is easily seen to be the field $\FF_q\laurentseries{t}$ (see \cite{HartlPinkVectorBundles}, Prop. 2.3).

Let $n\in\Z$.  The twisting sheaf $\OO(n)$ is a $\sigma$-bundle with underlying sheaf $\OO$.  The isomorphism $\tau_{\OO(n)}$ is defined as the composition of $\tau_{\OO}\from \sigma^*\OO_{D^*}\to \OO_{D^*}$ followed by multiplication by $t^{-n}$.

For $r\geq 1$, let $[r]$ denote the morphism
\[ [r]\from D^*\to D^*,\; t\mapsto t^r. \]
Then for an integer $d$ relatively prime to $r$ we set
\[ \mathcal{F}_{d,r}=[r]_*\OO(d), \]
together with the induced isomorphism $\tau_{\mathcal{F}_{d,r}}=[r]_*\tau_{\OO(d)}$.  Then $\mathcal{F}_{d,r}$ is a $\sigma$-bundle of rank $r$.

In keeping with the notation of \cite{FarguesFontaineCurve}, let us write $\OO(d/r)=\mathcal{F}_{d,r}$ whenever $d$ and $r$ are relatively prime integers with $r\geq 1$.  Then $\OO(\lambda)$ makes sense for any $\lambda\in\Q$.  We have $\OO(\lambda)\otimes\OO(\lambda')\isom\OO(\lambda+\lambda')$, and $\OO(\lambda)^*\isom \OO(-\lambda)$.

\begin{prop}{\cite{HartlPinkVectorBundles}, Prop. 8.4} \label{GlobalSectionsOfOlambda} Let $\lambda\in\Q$.  Then $H^0(\lambda)\neq 0$ if and only if $\lambda>0$.
\end{prop}

Hartl and Pink give a complete description of the category of $\sigma$-bundles.  In particular we have the following classification theorem of Dieudonn\'e-Manin type.

\begin{Theorem}{\cite{HartlPinkVectorBundles}, Thm. 11.1 and Cor. 11.8} \label{ClassificationOfSigmaBundles} Every $\sigma$-bundle $\mathcal{F}$ is isomorphic to one of the form $\bigoplus_{i=1}^{s} \OO(\lambda_i)$, with $\lambda_i\in\Q$.  Furthermore, $\mathcal{F}$ determines the $\lambda_i$ up to permutation.
\end{Theorem}

Since $Y_{F,E}^{\ad}\isom D^*$, with $\phi$ corresponding to $\sigma$, vector bundles on $X_{F,E}^{\ad}=Y_{F,E}^{\ad}/\phi^{\Z}$ correspond to $\sigma$-bundles, and we have a similar classification of them.

\subsection{$X_{F,E}^{\ad}$ is geometrically simply connected}

In \cite{FarguesFontaineCurve}, Thm. 18.1, it is shown that the scheme $X=X_{F,E}$ is geometrically connected, which is to say that every finite \'etale cover of $X_{\overline{E}}$ is split.  We need a similar result for the adic curve $X_{F,E}^{\ad}$:

\begin{Theorem} \label{XadIsSimplyConnected} $E'\mapsto X_{F,E}^{\ad}\otimes E'$ is an equivalence between the category of finite \'etale $E$-algebras and the category of finite \'etale covers of $X_{F,E}^{\ad}$.
\end{Theorem}

Unfortunately, Thm. \ref{XadIsSimplyConnected} cannot be deduced directly from the corresponding theorem about the scheme $X$.  Following a suggestion of Fargues, we can still mimic the proof of the simple-connectedness of $X$ in this adic context, by first translating the problem into characteristic $p$, and applying the Hartl-Pink classification of $\sigma$-bundles (Thm. \ref{ClassificationOfSigmaBundles}).

\begin{lemma}  \label{TracePairingLemma} Let $f\from Y\to X$ be a finite \'etale morphism of adic spaces of degree $d$.  Then $\mathcal{F}=f_*\OO_Y$ is a locally free $\OO_X$-module, and
\[ \left(\bigwedge^d\mathcal{F}\right)^{\otimes 2}\isom \OO_X. \]
\end{lemma}

\begin{proof}  If $A$ is a ring, and if $B$ is a finite \'etale $A$-algebra, then $B$ is flat and of finite presentation, hence locally free.  Furthermore, the trace map
\[ \tr_{B/A} \from B\otimes B\to A \]
is perfect, so that $B$ is self-dual as an $A$-module.  Globalizing, we get that $\mathcal{F}$ is locally free and $\mathcal{F}\isom \mathcal{F}^*=\Hom(\mathcal{F},\OO_X)$.  Taking top exterior powers shows that $\bigwedge^d\mathcal{F}\isom \bigwedge^d\mathcal{F}^*=\left(\bigwedge^d\mathcal{F}\right)^*$, so that the tensor square of $\bigwedge^d\mathcal{F}$ is trivial.
\end{proof}
%%Possibly helpful:  http://math.stanford.edu/~conrad/676Page/handouts/discexist.pdf
%http://stacks.math.columbia.edu/download/algebra.pdf

The same lemma appears in \cite{FarguesFontaineCurve}, Prop. 4.7.

Let $H_E$ be the Lubin-Tate formal $\OO_E$-module corresponding to the uniformizer $\pi$, so that $[\pi]_{H_E}(T)\equiv T^q\pmod {\pi}$.  Let $K$ be the completion of the field obtained by adjoining the torsion points of $H_E$ to $E$.  Then $K$ is a perfectoid field and $K^{\flat}=\FF_q\laurentseries{t^{1/q^\infty}}$.

\begin{prop} \label{XGeomConnectedCharp} Suppose $E$ has characteristic $p$.  Then $X_{F,E}^{\ad}$ is geometrically simply connected.  In other words, $E'\mapsto X_{F,E}^{\ad}\otimes E'$ is an equivalence between the category of finite \'etale $E$-algebras and finite \'etale covers of $X_{F,E}^{\ad}$.  
\end{prop}

\begin{proof}  It suffices to show that if $f\from Y\to X_{F,E}^{\ad}$ is a finite \'etale cover of degree $n$ with $Y$ geometrically irreducible, then $n=1$.  Given such a cover, let $\mathcal{F}=f_*\OO_Y$.   This is a sheaf of $\OO_{X_{F,E}^{\ad}}$-algebras, so we have a multiplication morphism $\mu\from \mathcal{F}\otimes\mathcal{F}\to\mathcal{F}$.  

Consider $\mathcal{F}$ as a $\sigma$-bundle.   By Thm. \ref{ClassificationOfSigmaBundles}, $\mathcal{F}\isom \bigoplus_{i=1}^n \OO(\lambda_i)$, for a collection of slopes $\lambda_i\in \Q$ (possibly with multiplicity).  Assume that $\lambda_1\geq \cdots \geq \lambda_n$.

The proof now follows that of \cite{FarguesFontaineCurve}, Th\'eor\`eme 18.1.  We claim that $\lambda_1\leq 0$.  Assume otherwise, so that $\lambda_1>0$.   After pulling back $\mathcal{F}$ through some $[d]\from D^{*}\to D^{*}$, we may assume that $\lambda_i\in\Z$ for $i=1,\dots,s$ 
(see \cite{HartlPinkVectorBundles}, Prop. 7.1(a)).    
The restriction of $\mu$ to $\OO(\lambda_1)\otimes\OO(\lambda_1)$ is the direct sum of morphisms
\[ \mu_{1,1,k}\from \OO(\lambda_1)\otimes\OO(\lambda_1) \to \OO(\lambda_k)\]
for $k=1,\dots,n$.  The morphism $\mu_{1,1,k}$ is tantamount to a global section of  
\[ \Hom(\OO(\lambda_1)^{\otimes 2},\OO(\lambda_k))=\OO(\lambda_k-2\lambda_1).\]  But since $\lambda_k-2\lambda_1<0$, Prop. \ref{GlobalSectionsOfOlambda} shows that $H^0(\OO(\lambda_k-2\lambda_1))=0$.  Thus the multiplication map $\OO(\lambda_1)\otimes\OO(\lambda_1)\to 0$ is 0.  This means that $H^0(Y,\OO_Y)$ contains zero divisors, which is a contradiction because $Y$ is irreducible.

Thus $\lambda_1\leq 0$, and thus $\lambda_i\leq 0$ for all $i$.  By Lemma \ref{TracePairingLemma}, $\left(\bigwedge^n\mathcal{F}\right)^{\otimes 2} \isom \OO_{D^*}$, from which we deduce $\sum_{i=1}^n\lambda_i=0$.  This shows that $\lambda_i=0$ for all $i$, and therefore $\mathcal{F}\isom\OO_X^n$.  We find that $E'=H^0(Y,\OO_Y)$ is an \'etale $E$-algebra of degree $n$.  Since $Y$ is geometrically irreducible, $E'\otimes_E E''$ must be a field for every separable field extension $E''/E$, which implies that $E'=E$ and $n=1$.  
\end{proof}

\begin{prop}  \label{GeometricallySimplyConnected} Suppose $E$ has characteristic 0.  Then $X_{F,E}^{\ad}$ is geometrically simply connected.
\end{prop}

\begin{proof}  Let $C$ be a complete algebraically closed field containing $E$.  We want to show that $X_{F,E}^{\ad}\hat{\otimes} C$ admits no nontrivial finite \'etale covers.  Such covers are equivalent to covers of its tilt, which is
\[ (X_{F,E}^{\ad}\hat{\otimes} C)^{\flat} \isom X_{F,L(E)}^{\ad}\hat{\otimes} C^{\flat}.\]
The latter is simply connected, by Prop. \ref{XGeomConnectedCharp} (note that $C^{\flat}$ is algebraically closed), which shows there are no nontrivial covers.
\end{proof}

% \begin{proof} The following categories are equivalent:
% \begin{enumerate}
% \item Finite \'etale covers of $X_{F,E}^{\ad}\hat{\otimes} \hat{E}_{\infty}$,
% \item Finite \'etale covers of $\left(X_{F,E}^{\ad}\hat{\otimes} \hat{E}_{\infty}\right)^{\flat}$,
% \item Finite \'etale covers of $X_{F,L(E)}^{\ad}\hat{\otimes} \hat{E}_{\infty}^{\flat}$,
% \item Finite \'etale covers of $X_{F,L(E)}^{\ad}$,
% \item Finite $L(E)$-algebras. 
% \item Finite \'etale $\hat{E}_{\infty}^{\flat}$-algebras.
% \item Finite \'etale $\hat{E}_{\infty}$-algebras.
% \item Finite \'etale $E_\infty$-algebras.
% \item The limit of the categories of finite $E_n$-algebras.
% \end{enumerate}

% Now suppose $Y\to X_{F,E}^{\ad}$ is a finite \'etale cover.
% \end{proof}

\section{Proof of the main theorem}
As before, $H_E$ is the Lubin-Tate formal $\OO_E$-module attached to the uniformizer $\pi$.  Let us recall the construction of $H_E$:  choose a power series $f(T)\in T\OO_E\powerseries{T}$ with $f(T)\equiv T^q\pmod{\pi}$.  Then $H_E$ is the unique formal $\OO_E$-module satisfying $[\pi]_{H_E}(T)=f(T)$.   Let $t=t_1,t_2,\dots$ be a compatible family of roots of $f(T),f(f(T)),\dots$, and let $E_n=E(t_n)$.  For each $n\geq 1$, $H_E[\pi^n]$ is a free $(\OO_E/\pi^n)$-module of rank 1, and the action of Galois induces an isomorphism $\Gal(E_n/E)\isom (\OO_E/\pi^n)^\times$.

Let $H_{E,0}=H_E\otimes_{\OO_E}\FF_q$.  

\begin{lemma}   \label{H0pi} For all $n\geq 2$ we have an isomorphism $H_{E,0}[\pi^{n-1}]\isom \Spec \OO_{E_n}/t$.  The inclusion $H_{E,0}[\pi^{n-1}]\to H_{E,0}[\pi^{n}]$ corresponds to the $q$th power Frobenius map $\Frob_q\from\OO_{E_{n+1}}/t \to \OO_{E_n}/t$.  
\end{lemma}

\begin{proof} We have $H_{E,0}[\pi^{n-1}]=\Spec \FF_q[T]/f^{(n-1)}(T)=\Spec \FF_q[T]/T^{q^{n-1}}$.  Meanwhile $\OO_{E_n}=\OO_{E_1}[T]/(f^{(n-1)}(T)-t)$, so that $\OO_{E_n}/t=\FF_q[T]/T^{q-1}$.  Thus $H_{E,0}[\pi^{n-1}]\isom \Spec \OO_{E_n}/t$.  The second claim in the lemma follows from $[\pi]_{H_{E,0}}(T)=T^q$.  
\end{proof}

Let $\tilde{H}_E$ be the {\em universal cover}:
\[ \tilde{H}_E=\varprojlim_\pi H_E.\]
Then $\tilde{H}_E$ is an $E$-vector space object in the category of formal schemes over $\OO_E$.  We will call such an object a {\em formal $E$-vector spaces}.   

Let $\tilde{H}_{E,0}=\tilde{H}_E\otimes_{\OO_E}\FF_q$, a formal $E$-vector space over $\FF_q$.  Since $\tilde{H}_{E,0}=\Spf \FF_q\powerseries{T}$ and $[\pi]_{H_{E,0}}(T)=T^q$, we have
\begin{eqnarray*}
\tilde{H}_{E,0}&=&\varprojlim_{\pi} \Spf \FF_q\powerseries{T} \\
&=&\Spf \left(\varinjlim_{T\mapsto T^q} \FF_q\powerseries{T}\right)^{\wedge}\\ 
&=&\Spf \FF_q\powerseries{T^{1/q^\infty}}. 
\end{eqnarray*}
In fact we also have $\tilde{H}_E=\Spf \OO_E\powerseries{T^{1/q^\infty}}$, see \cite{WeinsteinStableReduction}, Prop. 2.4.2(2).

%\begin{prop}  $\tilde{H}_E$ is an $E$-vector space object in the category of formal schemes over $\OO_E$.  We have $\tilde{H}_E\isom \Spf \OO_E\powerseries{T^{1/q^\infty}}$.  
%\end{prop}

%\begin{proof} See \cite{WeinsteinStableReduction}, Prop. 2.4.2.  
%\end{proof}

% Let $H_0=H\otimes_{\OO_E}\FF_q$.  

\begin{lemma} 
\label{tildeHE0} We have an isomorphism of formal $E$-vector spaces over $\FF_q$:
\[ \tilde{H}_{E,0}=\varinjlim_\pi \varprojlim_n H_{E,0}[\pi^n]. \]
\end{lemma}

\begin{proof} For each $n\geq 1$ we have the closed immersion $H_{E,0}[\pi^n]\to H_{E,0}$.  Taking inverse limits gives a map $\varprojlim_n H_{E,0}[\pi^n]\to \tilde{H}_{E,0}$, and taking injective limits gives a map $\varinjlim_\pi \varprojlim_n H_{E,0}[\pi^n]\to\tilde{H}_{E,0}$.  The corresponding homomorphism of topological rings is
\[ \FF_q\powerseries{T^{1/q^\infty}}\to \varprojlim_{T\mapsto T^q} \varinjlim \FF_q[T^{1/q^n}]/T =\varprojlim_{T\mapsto T^q} \FF_q[T^{1/q^\infty}]/T, \]
which is an isomorphism.
\end{proof}

\begin{prop} \label{H0AndEinfty} There exists an isomorphism of formal schemes over $\FF_q$:
\[ \tilde{H}_{E,0}\isom \Spf \OO_{\hat{E}_\infty^{\flat}}. \]
This isomorphism is $E^\times$-equivariant, where the action of $E^\times$ on $\OO_{\hat{E}_\infty^{\flat}}$ is defined as follows:  $\OO_E^\times$ acts through the isomorphism $\OO_E^\times\isom \Gal(E_\infty/E)$ of class field theory, and $\pi\in E^\times$ acts as the $q$th power Frobenius map.  
\end{prop}

\begin{proof}  Combining Lemmas \ref{H0pi} and \ref{tildeHE0}, we get
\begin{eqnarray*}
\tilde{H}_{E,0} &\isom& \varinjlim_\pi \varprojlim_n H_{E,0}[\pi^n]\\
&\isom& \varinjlim_{\Frob_q} \Spec \OO_{E_\infty}/t \\
&\isom& \Spf \OO_{E_\infty^\flat}.
\end{eqnarray*}
The compatibility of this isomorphism with the $\OO_E^\times$ follows from the definition of the isomorphism $\OO_E^\times\isom \Gal(E_\infty/E)$ of local class field theory.  The compatibility of the action of $\pi$ follows from the second statement in Lemma \ref{H0pi}.  
\end{proof}

Let $C$ be a complete algebraically closed field containing $E$.  Write $\tilde{H}^{\ad}_{E,C}$ for the generic fiber of $\tilde{H}\hat{\otimes}\OO_C$.   By Prop. \ref{TiltOfGenericFiber}, $\tilde{H}^{\ad}_{E,C}$ is a perfectoid space.

\begin{prop} \label{HTildeVersusY} There exists an isomorphism of adic spaces
%\[ \left(\tilde{H}^{\ad}_C\backslash\set{0})/\pi^\Z\right)^{\flat} \isom 
%\left(X
%_{C^{\flat},E}^{\ad} \hat{\otimes} \hat{E}_{\infty}\right)^{\flat}, \]
%%This is only true on the level of topological spaces, notthe structure sheaves.
\[ \left(\tilde{H}^{\ad}_{E,C}\backslash\set{0}\right)^{\flat}\isom \left(Y_{C^\flat,E}^{\ad}\hat{\otimes} \hat{E}_{\infty}\right)^{\flat}
\]
equivariant for the action of $\OO_E^\times$ (which acts on $E_\infty$ by local class field theory).  The action of $\pi\in E^\times$ on the left corresponds to the action of $\phi^{-1}\otimes 1$ on the right, up to composition with the absolute Frobenius morphism on $\left(Y_{C^{\flat},E}\hat{\otimes} \hat{E}_\infty\right)^{\flat}$. 
\end{prop}

% \begin{rmk}  In this isomorphism there are two tilted perfectoid spaces

 % \end{rmk}

\begin{proof}  By Prop. \ref{H0AndEinfty} we have an isomorphism $\tilde{H}_0\isom \Spf \OO_{\hat{E}_\infty^{\flat}}$.  Extending scalars to $\OO_C$,  taking the adic generic fiber and applying tilts gives an isomorphism (cf. Prop. \ref{TiltOfGenericFiber})
\begin{eqnarray*}
\left(\tilde{H}^{\ad}_{E,C}\right)^{\flat} 
&\isom & 
\Spf\left(\OO_C \hat{\otimes}_{\FF_q} \OO_{E_\infty}\right)^{\flat,\ad}_{C}  \\
&\isom &
\Spf \left(\OO_{C^{\flat}} \hat{\otimes}_{\FF_q} \OO_{\hat{E}_\infty^{\flat}} \right)^{\ad}_{C^{\flat}}
\end{eqnarray*}
The adic space $\Spf \left(\OO_{C^{\flat}} \hat{\otimes}_{\FF_q} \OO_{\hat{E}_\infty^{\flat}} \right)^{\ad}_{C^{\flat}}$ has two ``Frobenii'':  one coming from $\OO_{C^\flat}$ and the other coming from $\OO_{\hat{E}^{\flat}_\infty}$.  Their composition is the absolute $q$th power Frobenius.  The action of $\pi$ on $\tilde{H}^{\ad,\flat}_C$ corresponds to the Frobenius on $\OO_{\hat{E}^{\flat}_\infty}$.  Removing the ``origin'' from both sides of this isomorphism and using Prop. \ref{TiltOfYadNormField} gives the claimed isomorphism.

\end{proof}

\begin{lemma} \label{FrobEquivariant} Let $X$ be a perfectoid space which is fibered over $\Spec \FF_q$, and suppose $f\from X\to X$ is an $\FF_q$-linear automorphism.  Let $\Frob_q\from X\to X$ be the absolute Frobenius automorphism of $X$.  Then the category of $f$-equivariant finite \'etale covers of $X$ is equivalent to the category of $f\circ\Frob_q$-equivariant finite \'etale covers of $X$.
\end{lemma}

\begin{proof}  First observe that a perfectoid algebra in characteristic $p$ is necessarily perfect (\cite{ScholzePerfectoidSpaces}, Prop. 5.9), which implies that absolute Frobenius is an automorphism of any perfectoid space in characteristic $p$.  Then note that if $Y\to X$ is a finite \'etale cover, then $Y$ is also perfectoid (\cite{ScholzePerfectoidSpaces}, Thm. 7.9(iii)), so that $\Frob_q$ is an automorphism of $Y$.  

The proof of the lemma is now formal:  if $g\from Y\to X$ is a finite \'etale cover and $f_Y\from Y\to Y$ lies over $f\from X\to X$, then $f_Y\circ \Frob_q\from Y\to Y$ lies over $f\circ\Frob_q\from X\to X$.  Since $\Frob_q$ is invertible on $Y$, the functor is invertible.
\end{proof}

We can now prove the main theorem.
\begin{Theorem}  \label{MainTheorem} There is an equivalence between the category of $E^\times$-equivariant \'etale covers of $\tilde{H}_C^{\ad}\backslash\set{0}$ and the category of finite \'etale $E$-algebras.
\end{Theorem}

\begin{proof}
In the following chain of equivalences, ``$G$-cover of $X$'' is an abbreviation for ``$G$-equivariant finite \'etale cover of $X$''.     
\[\begin{matrix*}[l]
& &
\set{\text{$E^\times$-covers of $\tilde{H}_C^{\ad}\backslash\set{0}$}}
&
\\
& \isom &
\set{\text{$E^\times$-covers of 
$\tilde{H}_C^{\ad,\flat}\backslash\set{0}$}}
&\text{\cite{ScholzePerfectoidSpaces}, Thm. 7.12}
 \\
& \isom &
\set{\text{$\OO_E^\times\times (\phi^{-1}\circ\Frob_q)^\Z$-covers of $Y_{C^\flat,L(E)}^{\ad}\otimes\hat{E}^{\flat}_\infty$ }} 
&\text{Prop. \ref{HTildeVersusY}}
\\
& \isom &
\set{\text{$\OO_E^\times\times \phi^\Z$-covers of $Y_{C^\flat,L(E)}^{\ad}\otimes\hat{E}^{\flat}_\infty$ }} 
&\text{Lemma. \ref{FrobEquivariant}}
\\
& \isom &
\set{\text{$\OO_E^\times$-covers of $X_{C^\flat,L(E)}^{\ad}\otimes\hat{E}^{\flat}_\infty$}}
&\text{Defn. of $X_{C^\flat,L(E)}^{\ad}$} 
\\
& \isom &
\set{\text{$\OO_E^\times$-equivariant finite \'etale $\hat{E}^{\flat}_\infty$-algebras}} 
&\text{Prop. \ref{XGeomConnectedCharp}}
\\
& \isom &
\set{\text{$\OO_E^\times$-equivariant finite \'etale $\hat{E}_\infty$-algebras}}
&\text{Theory of norm fields}
\\
& \isom &
\set{\text{Finite \'etale $E$-algebras}}
&L\mapsto L^{\OO_E^\times}
\end{matrix*}
\]
\end{proof}

\subsection{Functoriality in $E$.}
Let us write
\[ Z_E = (\tilde{H}^{\ad}_{E,C}\backslash\set{0})/E^\times.\]
As in the introduction, $Z_E$ may be considered as a sheaf on $\text{Perf}_C$, although the reader may also interpret the definition of $Z_E$ purely formally, keeping in mind that a ``finite \'etale cover of $Z_E$'' means the same thing as an $E^\times$-equivariant finite \'etale cover of $\tilde{H}^{\ad}_{E,C}$.

Let us check that $Z_E$ really only depends on $E$.  The construction depends on the choice of Lubin-Tate formal $\OO_E$-module $H=H_E$, which depends in turn on the choice of uniformizer $\pi$.  If $\pi'\in E$ is a different uniformizer, with corresponding $\OO_E$-module $H'$, then $H$ and $H'$ become isomorphic after base extension to $\OO_{\hat{E}^{\nr}}$, the ring of integers in the completion of the maximal unramified extension of $E$.  Such an isomorphism is unique up to multiplication by $E^\times$.  Thus the adic spaces $\tilde{H}^{\ad}_C$ and $(\tilde{H}')^{\ad}_C$ are isomorphic, and we get a {\em canonical} isomorphism $(\tilde{H}^{\ad}_C\backslash\set{0})/E^\times \to (\tilde{H}^{\ad}_C\backslash\set{0})/E^\times$.  Thus $Z_E$ only depends on $E$.

It is also natural to ask whether the formation of $Z_E$ is functorial in $E$.  That is, given a finite extension $E'/E$ of degree $d$, there ought to be a ``norm'' morphism $N_{E'/E}\from Z_{E'}\to Z_{E}$, which makes the following diagram commute:
\begin{equation}
\label{Pi1Diagram}
\xymatrix{
\pi_1^{\et}(Z_{E'}) \ar[r]^{N_{E'/E}} \ar[d]_{\sim} & \pi_1^{\et}(Z_{E}) \ar[d]^{\sim} \\
\Gal(\overline{E}/E') \ar[r] & \Gal(\overline{E}/E).
}
\end{equation}
There is indeed such a norm morphism;  it is induced from the {\em determinant} morphism on the level of $\pi$-divisible $\OO_E$-modules.  The existence of exterior powers of such modules is the subject of \cite{Hedayatzadeh}.  Here is the main result we need\footnote{This result requires the residue characteristic to be odd, but we strongly suspect this is unnecessary. See \cite{ScholzeWeinsteinModuli}, \S6.4 for a construction of the determinant map (on the level of universal covers of formal modules) without any such hypothesis.} (Thm. 4.34 of \cite{Hedayatzadeh}):  let $G$ be a $\pi$-divisible $\OO_E$-module of height $h$ (relative to $E$) and dimension 1 over a noetherian ring $R$.  Then for all $r\geq 1$ there exists a $\pi$-divisible $\OO_E$-module $\bigwedge^r_{\OO_E} G$ of height ${h\choose r}$ and dimension ${h-1\choose r-1}$, together with a morphism $\lambda\from G^r\to \bigwedge^r_{\OO_E}G$ which satisfies the appropriate universal property.  In particular the determinant $\bigwedge^h_{\OO_E} G$ has height 1 and dimension 1.

Let $\pi'$ be a uniformizer of $E'$, and let $H'$ be a Lubin-Tate formal $\OO_{E'}$-module.  Then $H'[(\pi')^{\infty}]$ is a $\pi'$-divisible $\OO_{E'}$-module over $\OO_{E'}$ of height 1 and dimension 1.  By restriction of scalars, it becomes a $\pi$-divisible $\OO_E$-module $H'[\pi^\infty]$ of height $d$ and dimension 1.  Then $\bigwedge^d_{\OO_E} H'[\pi^\infty]$ is a $\pi$-divisible $\OO_E$-module of height 1 and dimension 1, so that it is the $\pi$-power torsion in a Lubin-Tate formal $\OO_E$-module $\bigwedge^d H'$ defined over $\OO_{E'}$.  For all $n\geq 1$ we have an $\OO_E/\pi^n$-alternating morphism
\[ \lambda\from  H'[\pi^n]^d \to \bigwedge^d H'[\pi^n] \]
of $\pi$-divisible $\OO_E$-modules over $\OO_{E'}$.  Let $H'_0=H'\otimes \OO_{E'}/\pi'$.  Reducing mod $\pi'$, taking inverse limits with respect to $n$ and applying Lemma \ref{tildeHE0} gives a morphism
\[ \lambda_0\from (\tilde{H}'_0)^d \to \widetilde{\bigwedge^d H'}_0 \]
of formal vector spaces over $\OO_{E'}/\pi'$.  By the crystalline property of formal vector spaces (\cite{ScholzeWeinsteinModuli}, Prop. 3.1.3(ii)), this morphism lifts uniquely to a morphism 
\[ \tilde{\lambda}\from (\tilde{H}')^d \to \widetilde{\bigwedge^d H'} \]
of formal vector space over $\OO_{E'}$.  

Since $\bigwedge^d H'$ and $H$ are both height 1 and dimension 1, they become isomorphic after passing to $\OO_C$.  Let $\alpha_1,\dots,\alpha_n$ be a basis for $E'/E$, and define a morphism of formal schemes
\begin{eqnarray*}
\tilde{H}'_{\OO_C} &\to& \widetilde{\bigwedge^d H'}_{\OO_\C} \isom \tilde{H}_{\OO_C}\\
x &\mapsto& \tilde{\lambda}(\alpha_1 x,\dots,\alpha_d x) 
\end{eqnarray*}
%For $\alpha\in \E^\times$ we have $\det(\alpha x_1,\alpha x_2,\dots,\alpha x_n)=N_{E'/E}(\alpha)\det(x_1,\dots,x_n)$. 
After passing to the generic fiber we get a well-defined map $N_{E'/E}\from Z_{E'}\to Z_E$ which does not depend on the choice of basis for $E'/E$.  

The commutativity of the diagram in Eq. \eqref{Pi1Diagram} is equivalent to the following proposition.
\begin{prop}  The following diagram commutes:
\[
\xymatrix{
\set{\text{\'Etale $E$-algebras}} \ar[r]^{A\mapsto A\otimes_E E'} \ar[d] & \set{\text{\'Etale $E'$-algebras}} \ar[d] \\
\set{\text{\'Etale covers of $Z_E$}} \ar[r] & \set{\text{\'Etale covers of $Z_E'$}}.
}
\]
Here the bottom arrow is pullback via $N_{E'/E}$.   
\end{prop}

\begin{proof} Ultimately, the proposition will follow from the functoriality of the isomorphism in Lemma \ref{H0pi}.  For $n\geq 1$, let $E'_n$ be the field obtained by adjoining the $\pi^n$-torsion in $H'$ to $E'$.  (Note that the $\pi^n$-torsion is the same as the $(\pi')^{en}$ torsion, where $e$ is the ramification degree of $E'/E$.)  The existence of $\lambda$ shows that $E_n$ contains the field obtained by adjoining the $\pi^n$-torsion in $\bigwedge^d H'$ to $E_n$.  Namely, let $x\in H[\pi^n](\OO_{E'_n})$ be a primitive element, in the sense that $x$ generates $H'[\pi^n](\OO_{E'_n})$ as an $\OO_{E'}/\pi^n$-module.  If $\alpha_1,\dots,\alpha_d$ is a basis for $\OO_{E'}/\OO_E$ then $\lambda(\alpha_1 x,\dots,\alpha_d x)$ generates $\bigwedge H'[\pi^n](\OO_{E'_n})$ as an $\OO_{E}/\pi^n$-module.  Since $\bigwedge^d H'$ and $H$ become isomorphic over $\hat{E}^{\prime,\nr}$, we get a compatible family of embeddings $\hat{E}^{\nr}_n\injects \hat{E}^{\prime,\nr}_n$.  

By construction, these embeddings are compatible with the isomorphisms in Lemma \ref{H0pi}, so that the following diagram commutes:
\[
\xymatrix{
\Spf \OO_{E_n^{\prime,\nr}/t'} \ar[r]^{\sim} \ar[d] & \tilde{H}'_0[\pi^{n-1}]\otimes\overline{\FF}_q \ar[d]^{N_{E'/E}} \\
\Spf \OO_{E_n^{\nr}/t} \ar[r]_{\sim} & \tilde{H}_0[\pi^{n-1}]\otimes\overline{\FF}_q
}
\]
Here $t'\in \OO_{E_1}$ is a uniformizer.  From here we get the commutativity of the following diagram:
\[
\xymatrix{
\Spf \left(\OO_{\hat{E}_\infty^{\prime,\flat}}\hat{\otimes}_{\OO_{E'}/\pi'} \OO_{C^\flat}\right) \ar[r] \ar[d] & \tilde{H}_{\OO_C}^{\prime,\flat} \ar[d]^{N_{E'/E}} \\
\Spf \left(\OO_{\hat{E}_\infty^{\flat}}\hat{\otimes}_{\FF_q} \OO_{C^\flat}\right) \ar[r] & \tilde{H}_{\OO_C}^{\flat}. 
}
\]
One can now trace this compatibility with the chain of equivalences in the proof of Thm. \ref{MainTheorem} to get the proposition.  The details are left to the reader.
\end{proof}

\subsection{Descent from $C$ to $E$}  The object $Z_E=(\tilde{H}^{\ad}_{E,C}\backslash\set{0})/E^\times$ is defined over $C$, but it has an obvious model over $E$, namely
$Z_{E,0}=(\tilde{H}^{\ad}_E\backslash\set{0})/E^\times$.   Since the geometric fundamental group of $Z_{E,0}$ is $\Gal(\overline{E}/E)$, we have an exact sequence
\[ 1\to \Gal(\overline{E}/E)\to \pi_1^{\et}(Z_{E,0})\to \Gal(\overline{E}/E) \to 1. \]
This exact sequence splits as a direct product.  In other words:
\begin{prop}  Let $E'/E$ be a finite extension, and let $Z_E'\to Z_E$ be the corresponding finite \'etale cover under Thm. \ref{MainTheorem}.  Then $Z_E'$ descends to a finite \'etale cover $Z_{E,0}'\to Z_{E,0}$.  
\end{prop}

\begin{proof}  Let us write $R=H^0(\tilde{H},\OO_{\tilde{H}})$ for the coordinate ring of $\tilde{H}$, so that $R\isom \OO_E\powerseries{t^{1/q^\infty}}$.  Then $R/\pi\isom\FF_q\powerseries{t^{1/q^\infty}}$ is a perfect ring and $R=W_{\OO_E}(R/\pi)$.  By Prop. \ref{H0AndEinfty} we have $R/\pi=\OO_{\hat{E}_{\infty}^{\flat}}$.  Thus $Z_{E,0}=\Spf W(\OO_{\hat{E}_\infty^{\flat}})^{\ad}_E\backslash\set{0}$.  

Now suppose $E'/E$ is a finite extension.  Let $\hat{E}'_{\infty}=E'\otimes_E \hat{E}_\infty$, a finite \'etale $\hat{E}_\infty$-algebra.  Then $\hat{E}'_{\infty}$ is a perfectoid $\hat{E}_\infty$-algebra (in fact it a product of perfectoid fields), so that we can form the tilt $\hat{E}^{\prime,\flat}_\infty$, a perfect ring containing $\hat{E}_\infty^{\flat}$.  Let $\OO_{\hat{E}^{\prime,\flat}_\infty}$ be its subring of power-bounded elements.  We put
\[ Z_{E,0}'=\Spf W_{\OO_E}(\OO_{\hat{E}^{\prime,\flat}})^{\ad}_E\backslash\set{0}, \]
a finite \'etale cover of $Z_{E,0}$.  

We claim that $Z_E'=Z_{E,0}\hat{\otimes} C$.  It suffices to show that the tilts are the same.  We have
\begin{eqnarray*}
\left(Z_{E,0}\hat{\otimes} C\right)^{\flat}
&=&
\Spf \left(W_{\OO_E}(\OO_{\hat{E}^{\prime,\flat}}\hat{\otimes}\OO_C)^{\flat,\ad}_{C^{\flat}}\backslash\set{0}\right)/E^\times \\
&=& \left(\Spf \left(\OO_{\hat{E}^{\prime,\flat}} \hat{\otimes}\OO_{C^{\flat}}\right)^{\ad}_{C^{\flat}}\backslash\set{0}\right)/E^\times, 
\end{eqnarray*}
which by construction is the tilt of $Z_E'$.  
\end{proof}

\bibliographystyle{amsalpha}
\bibliography{bibfile}

\end{document}